\documentclass[draft]{amsart}
\usepackage{amsmath}
\usepackage{amssymb}
\usepackage{amsthm}
\usepackage[msc-links,abbrev]{amsrefs}
\usepackage{hyperref}
\usepackage[noabbrev,capitalize]{cleveref}
\usepackage{mathrsfs}
\usepackage{mathtools}
\usepackage{slashed}
\usepackage{esint}

\DeclareMathOperator{\tr}{tr}
\DeclareMathOperator{\Vol}{Vol}
\DeclareMathOperator{\dvol}{dV}

\DeclareMathOperator{\Ric}{Ric}

\newcommand{\defn}[1]{{\boldmath\bfseries#1}}

\newcommand{\cg}{\widetilde{g}}

\newcommand{\hg}{\widehat{g}}

\newcommand{\lp}{\langle}
\newcommand{\rp}{\rangle}
\newcommand{\lv}{\lvert}
\newcommand{\rv}{\rvert}
\newcommand{\lV}{\lVert}
\newcommand{\rV}{\rVert}

\newcommand{\bCP}{\mathbb{C}P}

\newcommand{\mF}{\mathcal{F}}

\newcommand{\bR}{\mathbb{R}}

\newcommand{\nablas}{\slashed{\nabla}}

\def\sideremark#1{\ifvmode\leavevmode\fi\vadjust{\vbox to0pt{\vss
 \hbox to 0pt{\hskip\hsize\hskip1em
 \vbox{\hsize3cm\tiny\raggedright\pretolerance10000
 \noindent #1\hfill}\hss}\vbox to8pt{\vfil}\vss}}}

\newcommand{\suchthat}{\mathrel{}:\mathrel{}}

\newtheorem{theorem}{Theorem}[section]
\newtheorem{proposition}[theorem]{Proposition}
\newtheorem{lemma}[theorem]{Lemma}
\newtheorem{corollary}[theorem]{Corollary}

\theoremstyle{definition}

\theoremstyle{remark}

\numberwithin{equation}{section}

\usepackage{tikz-cd}

\makeatletter
\@namedef{subjclassname@2020}{%
  \textup{2020} Mathematics Subject Classification}
\makeatother

\begin{document}

\title{The Obata--V\'etois argument and its applications}
\author{Jeffrey S. Case}
\address{109 McAllister Building \\ Penn State University \\ University Park, PA 16802 \\ USA}
\email{jscase@psu.edu}
\begin{abstract}
 We simplify V\'etois' Obata-type argument and use it to identify a closed interval $I_n$, $n \geq 3$, containing zero such that if $a \in I_n$ and $(M^n,g)$ is a closed conformally Einstein manifold with nonnegative scalar curvature and $Q_4 + a\sigma_2$ constant, then it is Einstein.
 We also relax the scalar curvature assumption to the nonnegativity of the Yamabe constant under a more restrictive assumption on $a$.
 Our results allow us to compute many Yamabe-type constants and prove sharp Sobolev inequalities on closed Einstein manifolds with nonnegative scalar curvature.
 In particular, we show that closed locally symmetric Einstein four-manifolds with nonnegative scalar curvature extremize the functional determinant of the conformal Laplacian, partially answering a question of Branson and {\O}rsted.
\end{abstract}
\maketitle

\section{Introduction}
\label{sec:intro}

Obata~\cite{Obata1971} proved that every closed conformally Einstein manifold $(M^n,g)$ with constant scalar curvature is Einstein.
There are two key points to his argument.
First, the constant scalar curvature assumption implies that $E := P - \frac{1}{n}Jg$ is divergence-free, where $P$ is the Schouten tensor and $J$ is its trace.
Second, the conformally Einstein assumption implies that there is a constant $\lambda \in \bR$ and a positive function $u \in C^\infty(M)$ such that
\begin{equation}
 \label{eqn:conformally-einstein}
 \nabla^2u = -uP + \frac{1}{2}u^{-1}\left( \lv\nabla u\rv^2 + \lambda \right)g ;
\end{equation}
i.e.\ $P^{u^{-2}g} = \frac{\lambda}{2}u^{-2}g$.
Integration by parts yields
\begin{equation}
 \label{eqn:obata}
 \int_M u\lv E\rv^2 \dvol = -\int_M \lp E, \nabla^2u \rp \dvol = 0 ,
\end{equation}
and hence $E=0$.
That is, $g$ is Einstein.

It is natural to ask whether a similar statement holds for other scalar Riemannian invariants.
The most satisfying such result was recently given by V\'etois~\cite{Vetois2022}: Any closed conformally Einstein manifold with positive Yamabe constant and constant fourth-order $Q$-curvature is itself Einstein.
Earlier partial results were given under the additional assumption of local conformal flatness:
Gursky~\cite{Gursky1997} classified the critical points of certain functional determinants in the conformal class of the round four-sphere, and Viaclovsky~\cite{Viaclovsky2000} classified locally conformally flat, conformally Einstein metrics of constant $\sigma_k$-curvature in the elliptic $k$-cones.
Outside the setting of scalar Riemannian invariants, there are also Obata-type results for the Boundary Yamabe Problem~\cite{Escobar1988} and for the CR Yamabe Problem~\cite{JerisonLee1988}.

V\'etois' proof is in the spirit of Obata's argument, but focused directly on positive solutions of the PDE $L_4u = \lambda u^p$, $p \in \bigl( 1, \frac{2n}{n-4} \bigr]$, and its four-dimensional analogue, where $L_4$ is the Paneitz operator~\cite{Paneitz1983}.
When $p = \frac{2n}{n-4}$, solutions of this PDE give rise to metrics of constant $Q$-curvature.
The benefit of his approach is that it classifies solutions of the subcritical problem (cf.\ \cites{Bidaut-VeronVeron1991,GidasSpruck1981}).
The downsides of his approach are that it relies heavily on the fact that the $Q$-curvature prescription problem, unlike the $\sigma_2$-curvature prescription problem, is \emph{semilinear}, and that finding the correct analogue of Equation~\eqref{eqn:obata} requires solving an undetermined coefficients problem with twelve unknowns.

The first goal of this article is to give a streamlined version of V\'etois' argument which both generalizes to certain fully nonlinear problems and requires solving an undetermined coefficients problem with only three unknowns.
To that end, given a Riemannian manifold $(M^n,g)$ and constants $a,b \in \bR$, define
\begin{equation}
 \label{eqn:defn-Ia}
 I_{a,b} := Q + a\sigma_2 + b\lv W \rv^2 .
\end{equation}
Here
\begin{align*}
 Q & := -\Delta J - 2\lv P\rv^2 + \frac{n}{2}J^2 , \\
 \sigma_2 & := \frac{1}{2}\left( J^2 - \lv P\rv^2 \right) , \\
 \lv W\rv^2 & := W_{abcd}W^{abcd} ,
\end{align*}
are the $Q$-curvature, the $\sigma_2$-curvature, and the squared length of the Weyl tensor, with the convention $-\Delta \geq 0$.

The significance of the family $I_{a,b}$ comes from the consideration of natural Riemannian scalar invariants which are variational within conformal class.
Let $I$ be a natural scalar Riemannian invariant\footnote{A \defn{natural scalar Riemannian invariant} is a linear combination of complete contractions of tensor products of $g$, its inverse, and covariant derivatives of its Riemann curvature tensor.} which  is homogeneous of degree $-4$;
i.e.\ $I^{c^2g}=c^{-4}I^g$ for any constant $c>0$ and any metric $g$.
Then $I$ is in the span of $\{ \Delta J, J^2, \lv P\rv^2, \lv W\rv^2 \}$.
Such an element is \defn{conformally variational} if there is a natural Riemannian functional\footnote{A functional $\mF \colon [g] \to \bR$ is \defn{natural} if $\mF(\Phi^\ast g) = \mF(g)$ for any diffeomorphism $\Phi$ of $M$.} $\mF \colon [g] \to \bR$ such that
\begin{equation*}
 \left. \frac{d}{dt} \right|_{t=0} \mF(e^{2tu}\hg) = \int_M uI^{\hg} \dvol_{\hg}
\end{equation*}
for any $\hg \in [g]$ and any $u \in C^\infty(M)$.
We call $\mF$ a \defn{conformal primitive} of $I$.
The subspace of conformally variational invariants which are homogeneous of degree $-4$ is three-dimensional~\cites{BransonOrsted1988,CaseLinYuan2016} and spanned by $\{ Q, \sigma_2, \lv W\rv^2 \}$.
Since $\sigma_2$ and $\lv W\rv^2$ depend on at most second-order derivatives of the conformal factor, the $I_{a,b}$-curvatures thus represent all such conformally variational invariants which are fourth-order in the conformal factor.

Our generalization of V\'etois' argument applies to the invariants $I_{a,0}$ with $a$ suitably close to zero:

\begin{theorem}
 \label{main-thm}
 Let $(M^n,g)$, $n \geq 3$, be a closed conformally Einstein manifold with nonnegative scalar curvature.
 Suppose additionally that there is an
 \begin{equation}
  \label{eqn:a-range}
  a \in \left[ \frac{n^2-7n+8 - \sqrt{n^4+2n^3-3n^2}}{2(n-1)} , \frac{n^2-7n+8 + \sqrt{n^4+2n^3-3n^2}}{2(n-1)} \right]
 \end{equation}
 such that $I_{a,0}$ is constant;
 if $J^g=0$, then assume also that $a$ is in the interior of this interval.
 Then $(M^n,g)$ is Einstein.
\end{theorem}

Note that $a=0$ satisfies Condition~\eqref{eqn:a-range}.
We do not know whether this condition is optimal.
It would be interesting to know whether the proof of \cref{main-thm} can be modified to allow $b \not= 0$ for manifolds which are conformal to a locally symmetric Einstein manifold.

The requirement in \cref{main-thm} that $g$ have nonnegative scalar curvature is somewhat unsatisfying.
One can relax this to an assumption on the Yamabe constant if one further restricts Condition~\eqref{eqn:a-range}.

\begin{theorem}
 \label{main-thm-q}
 Let $(M^n,g)$, $n \geq 3$, be a conformally Einstein manifold with nonnegative Yamabe constant.
 Suppose additionally that there is an
 \begin{equation}
  \label{eqn:restricted-a-range}
  a \in \{ 0 \} \cup \left[ \frac{n^2-7n+8 - \sqrt{n^4+2n^3-3n^2}}{2(n-1)} , -\frac{2(n-2)}{n-1} \right]
 \end{equation}
 such that $I_{a,0}$ is constant;
 if $n=3$, then assume also that $I_{a,0}\geq0$.
 Then $(M^n,g)$ is Einstein.
\end{theorem}

The case $a=0$ of \cref{main-thm-q} is the main result of V\'etois~\cite{Vetois2022}.
Direct computation shows that in dimension four, Condition~\eqref{eqn:restricted-a-range} reduces to
\begin{equation*}
 a \in \left[ -\frac{2+2\sqrt{21}}{3} , -\frac{4}{3} \right] ,
\end{equation*}
and hence \cref{main-thm-q} recovers Gursky's result~\cite{Gursky1997} while also removing the locally conformally flat assumption.
Neither \cref{main-thm} nor \cref{main-thm-q} recovers Viaclovsky's result~\cite{Viaclovsky2000}.

In both cases of \cref{main-thm-q}, one first shows that the assumptions force $g$ to have nonnegative scalar curvature, then one applies \cref{main-thm}.
The two cases are distinguished by their proof:
V\'etois' case $a=0$ is proven using the semilinearity of the constant $Q$-curvature equation and an adaptation of an argument of Gursky and Malchiodi~\cite{GurskyMalchiodi2014}.
For the other case, one observes that $I_{a,0} \geq 0$, and hence that $L_2J \geq 0$, where $L_2$ denotes the conformal Laplacian, and then applies an observation of Gursky~\cite{Gursky1998}.

We do not know if Condition~\eqref{eqn:restricted-a-range} is optimal.
However, it cannot be completely removed:
Gursky and Malchiodi~\cite{GurskyMalchiodi2012} gave an explicit example of a linear combination $I_{a,0}$ for which there is a locally conformally flat metric on $S^4$ with $I_{a,0}$ constant but which is not Einstein.

We expect that the strategy used to prove \cref{main-thm} can be adapted to other settings, and to that end we give a heuristic outline of its proof.
The key insight, implicit in V\'etois' argument,  is that Obata's argument \emph{only} requires finding a natural Riemannian symmetric $(0,2)$-tensor $T$ such that if $(M^n,g)$ has constant $I_{a,0}$-curvature, then
\begin{equation}
 \label{eqn:obata-strategy}
 \int_M \left( \lp \delta T, du \rp - u\lp T,P \rp + \frac{1}{2}u^{-1}\left( \lv\nabla u\rv^2 + \lambda \right)\tr T \right) \dvol \geq 0 
\end{equation}
with equality if and only if $g$ is Einstein.
In particular, unlike previous Obata-type results for specific linear combinations of the $Q$- and $\sigma_2$-curvatures~\cites{Gursky1997,Viaclovsky2000,ChangGurskyYang2003b}, it is \emph{not} necessary to choose $T$ to be trace- and divergence-free.
It is, however, advantageous for $T$ to be zero when evaluated at an Einstein metric.

The heuristic leading to a suitable tensor $T$ is as follows:
Since $Q$ is homogeneous of degree $-4$, one should ask the same of $T$.
Modulo multiples of the metric, the space of such natural Riemannian symmetric $(0,2)$-tensors is six-dimensional and spanned by $\{ JP, P^2, \nabla^2J, B, \check{W}, WP \}$, where $\nabla^2$ is the Hessian, $B$ is the Bach tensor, and $\check{W}$ and $WP$ are nontrivial partial contractions of $W \otimes W$ and $W \otimes P$.
The divergences of all but $JP$ and $\nabla^2J$ algebraically depend on $\nabla P$ or $W$, which are difficult to use to verify Estimate~\eqref{eqn:obata-strategy}.
Therefore we restrict our attention to $JP$, $\nabla^2J$, and suitable multiples of the metric.
Both $JP - \frac{1}{n}J^2g$ and $\nabla^2J - (\Delta J)g$ vanish at Einstein metrics.
Additionally, there is a constant $c$ such that $JP - cI_{a,0}g$ vanishes at Einstein metrics.
One then searches for a linear combination $T$ of these three tensors which satisfies Estimate~\eqref{eqn:obata-strategy}.
The desired tensor is (proportional to)
\begin{multline*}
 T := -\frac{2(n^2+2n-4+(n-1)a)}{n-1}JP + \frac{2n}{n-1}\nabla^2J \\
  + \left( - \frac{2n}{n-1}\Delta J + \frac{n^3+n^2-4+(n^2-1)a}{n(n-1)}J^2g - 2I_{a,0}g\right)g .
\end{multline*}

Recall another well-known result of Obata~\cite{Obata1971}:
If $(M^n,g)$ is a closed Einstein manifold and if $u \in C^\infty(M)$ is such that $e^{2u}g$ is Einstein, then either $u$ is constant or both $(M^n,g)$ and $(M^n,e^{2u}g)$ are homothetic to a round sphere.
Thus \cref{main-thm,main-thm-q} are essentially uniqueness results.

The second goal of this article is to use V\'etois' result to compute various Yamabe-type constants on closed Einstein manifolds with nonnegative scalar curvature.
This relies on the fact~\cites{GurskyMalchiodi2014,HangYang2004,HangYang2016t,ChangYang1995} that minimizers of the $Q$-Yamabe Problem exist on such manifolds, and hence are known by \cref{main-thm-q}.
A convexity argument, first used by Branson, Chang, and Yang~\cite{BransonChangYang1992} to study extremals of functional determinants, allows us to compute the $I_{a,b}$-Yamabe constant for a wide range of values $a$ and $b$.
We also use conformal invariance to express these as sharp Sobolev- and Onofri-type inequalities on closed Einstein manifolds with nonnegative scalar curvature.
Due to dimensional differences in the Sobolev embedding theorem, we separately discuss the cases $n \geq 5$, $n=3$, and $n=4$.

On manifolds of dimension at least five, the $I_{a,b}$-Yamabe constant is defined by minimizing the volume-normalized total $I_{a,b}$-curvature.
We compute this constant for Einstein manifolds with nonnegative scalar curvature.

\begin{theorem}
 \label{yamabe-dim5+}
 Let $(M^n,g)$, $n \geq 5$, be a closed Riemannian manifold such that $\Ric = (n-1)\lambda g \geq 0$.
 Pick constants $a \in [-4,0]$ and $b \leq 0$;
 if $b<0$, assume additionally that $\lv W\rv^2$ is constant.
 Set
 \begin{equation}
  \label{eqn:yamabe-dim5+-constant}
  C_{g,a,b} := \left( \frac{n(n^2-4)}{8}\lambda^2 + \frac{n(n-1)}{8}a\lambda^2 + b\lv W\rv^2 \right)\Vol(M)^{\frac{4}{n}} .
 \end{equation}
 Then
 \begin{equation}
  \label{eqn:yamabe-dim5+}
  \inf_{\hg \in [g]} \left\{ \int_M I_{a,b}^{\hg} \dvol_{\hg} \suchthat \Vol_{\hg}(M) = 1 \right \} = C_{g,a,b} .
 \end{equation}
 Moreover, $\hg \in [g]$ is extremal if and only if it is Einstein.
\end{theorem}

In particular, \cref{yamabe-dim5+} applies to all closed locally symmetric Einstein manifolds with nonnegative scalar curvature.

Since $I_{a,b}$ is conformally variational, there is a natural formally self-adjoint conformally covariant polydifferential operator whose constant term is a multiple of the $I_{a,b}$-curvature~\cite{CaseLinYuan2018b}.
This operator allows us to express Equation~\eqref{eqn:yamabe-dim5+} as a sharp functional inequality.

\begin{corollary}
 \label{sobolev-dim5+}
 Let $(M^n,g)$, $n \geq 5$, be a closed Riemannian manifold such that $\Ric = (n-1)\lambda g \geq 0$.
 Pick constants $a \in [-4,0]$ and $b \leq 0$;
 if $b<0$, assume additionally that $\lv W\rv^2$ is constant.
 Then
 \begin{align*}
  \MoveEqLeft \frac{n-4}{2}C_{g,a,b}\lV u \rV_{\frac{4n}{n-4}}^2 \leq \int_M \Biggl( (\Delta u^2)^2 - \frac{16}{(n-4)^2}a\lv\nabla u\rv^4 - \frac{4}{n-4}a\lv\nabla u\rv^2\Delta u^2 \\
   & \quad + \left( \frac{n^2-2n-4}{2} + \frac{n-1}{2}a\right)\lambda\lv\nabla u^2\rv^2 \\
   & \quad + \biggl( \frac{\Gamma\bigl(\frac{n+4}{2}\bigr)}{\Gamma\bigl(\frac{n-4}{2}\bigr)}\lambda^2 + \frac{n(n-1)(n-4)}{16}a\lambda^2 + \frac{n-4}{2}b\lv W\rv^2 \biggr) u^4 \Biggr) \dvol
 \end{align*}
 for all $u \in C^\infty(M)$, with equality if and only if either $u$ is constant or $(M^n,g)$ is homothetic to the round $n$-sphere and $u(\xi) = a( 1 + \xi \cdot \zeta )^{-\frac{n-4}{4}}$ for some constant $a \in \bR$ and some point $\zeta \in B_1(0)$.
\end{corollary}

The final conclusion of \cref{sobolev-dim5+} uses the homothety to regard $(M^n,g)$ as the round sphere $S^n = \partial B_1(0)$ where $B_1(0) \subset \bR^{n+1}$ is the unit ball.
In the special case $a=b=0$, taking $v=u^2$ in \cref{sobolev-dim5+} yields the sharp Sobolev inequality
\begin{multline*}
 \int_M \left( (\Delta v)^2 + \frac{n^2-2n-4}{2}\lambda\lv\nabla v\rv^2 + \frac{\Gamma\bigl(\frac{n+4}{2}\bigr)}{\Gamma\bigl(\frac{n-4}{2}\bigr)}\lambda^2 v^2 \right) \dvol \\
  \geq \frac{\Gamma\bigl(\frac{n+4}{2}\bigr)}{\Gamma\bigl(\frac{n-4}{2}\bigr)}\lambda^2\Vol_g(M)^{\frac{4}{n}}\lV v \rV_{\frac{2n}{n-4}}^2
\end{multline*}
on closed Einstein manifolds $(M^n,g)$, $n \geq 5$, with $\Ric = (n-1)\lambda g \geq 0$.

The $Q$-Yamabe Problem in dimension three has been studied by Hang and Yang~\cites{HangYang2004,HangYang2016t}.
While this problem still involves minimizing the normalized Paneitz energy, it is now equivalent to \emph{maximizing} the volume-normalized total $Q$-curvature.
This affects the three-dimensional analogue of \cref{yamabe-dim5+}.

\begin{theorem}
 \label{yamabe-dim3}
 Let $(M^3,g)$ be a closed Riemannian three-manifold such that $\Ric = 2\lambda g \geq 0$.
 Pick a constant $a \in [-4,0]$.
 Set
 \begin{equation}
  \label{eqn:yamabe-dim3-constant}
  C_{g,a} := \left( \frac{15}{8}\lambda^2 + \frac{3}{4}a\lambda^2 \right)\Vol(M)^{\frac{4}{3}} .
 \end{equation}
 Then
 \begin{equation*}
   \sup_{\hg \in [g]} \left\{ \int_M I_{a,0}^{\hg} \dvol_{\hg} \suchthat \Vol_{\hg}(M) = 1 \right \} = C_{g,a} .
 \end{equation*}
 Moreover, $\hg \in [g]$ is extremal if and only if it is Einstein.
\end{theorem}

The parameter $b$ is irrelevant in \cref{yamabe-dim3} because the Weyl tensor vanishes in dimension three.

Analogous to dimension at least five, \Cref{yamabe-dim3} gives rise to a functional inequality on closed Einstein three-manifolds with nonnegative scalar curvature.

\begin{corollary}
 \label{sobolev-dim3}
 Let $(M^3,g)$ be a closed Riemannian three-manifold such that $\Ric = 2\lambda g \geq 0$.
 Pick a constant $a \in [-4,0]$.
 Then
 \begin{multline*}
  -\frac{1}{2}C_{g,a}\lV u \rV_{}^2 \leq \int_M \biggl( (\Delta u^2)^2 - 16a\lv\nabla u\rv^4 + 4\lv\nabla u\rv^2\Delta u^2 \\
   + \left( a - \frac{1}{2}\right)\lambda\lv\nabla u^2\rv^2 - \frac{3(5+2a)}{16}\lambda^2 u^4 \biggr) \dvol_g
 \end{multline*}
 for all $u \in C^\infty(M)$, with equality if and only if either $u$ is constant or $(M^3,g)$ is homothetic to the round three-sphere and $u(\xi) = a( 1 + \xi \cdot \zeta )^{\frac{1}{4}}$ for some constant $a \in \bR$ and some point $\zeta \in B_1(0)$.
\end{corollary}

The special feature of the cases $n\not=4$ is that $\int I_{a,b} \dvol$ is a conformal primitive for the $I_{a,b}$-curvature.
This is not true in dimension four.
Instead, the functionals
\begin{align*}
 I(u) & := 4\int_M \lv W\rv^2u \dvol - \left( \int_M \lv W\rv^2 \dvol \right) \log \fint e^{4u}\dvol , \\
 II(u) & := \int_M uL_4u \dvol + 2\int_M Qu \dvol - \frac{1}{2}\left( \int_M Q \dvol \right) \log \fint e^{4u} \dvol , \\
 III(u) & := \int_M \left( 12\left(\Delta u + \lv\nabla u\rv^2\right)^2 - 4R\lv\nabla u\rv^2 - 4u\Delta R \right) \dvol ,
\end{align*}
can be used to produce conformal primitives\footnote{These are the volume-normalized versions of the conformal primitives for $\lv W\rv^2$, $Q$, and $-\Delta J$.}.
Indeed, given $\gamma_1,\gamma_2,\gamma_3 \in \bR$, the functional $\mF_{\gamma_1,\gamma_2,\gamma_3} \colon C^\infty(M) \to \bR$,
\begin{equation*}
 F_{\gamma_1,\gamma_2,\gamma_3} := \gamma_1 I + \gamma_2 II + \gamma_3 III ,
\end{equation*}
has the property~\cite{ChangYang1995} that $u \in C^\infty(M)$ is a critical point of $F_{\gamma_1,\gamma_2,\gamma_3}$ if and only if $I_{\gamma_1,\gamma_2,\gamma_3}^{e^{2u}g}$ is constant, where
\begin{equation*}
 I_{\gamma_1,\gamma_2,\gamma_3} := \gamma_1 \lv W \rv^2 + \left( \frac{\gamma_2}{2} + 6\gamma_3 \right)Q - 24\gamma_3 \sigma_2 .
\end{equation*}

The labeling $I$, $II$, and $III$ was introduced by Chang and Yang~\cite{ChangYang1995} in their study of the existence of extremal metrics for functional determinants on closed conformal four-manifolds.
More precisely, Branson and {\O}rsted~\cite{BransonOrsted1991b} showed that if $A^g$ is an integer power of a natural, formally self-adjoint, conformally covariant differential operator on a closed Riemannian four-manifold $(M^n,g)$ for which $\ker A^g \subseteq \bR$, then there are constants $\gamma_1,\gamma_2,\gamma_3 \in \bR$ such that
\begin{equation}
 \label{eqn:defn-Fa}
 F_A(u) := 720\pi^2\log\frac{\det A^{g_u}}{\det A^g} = F_{\gamma_1,\gamma_2,\gamma_3}(u)
\end{equation}
for all $u \in C^\infty(M)$ such that $g_u := e^{2u}g$ has $\Vol_{g_u}(M)=\Vol_g(M)$, where $\log \det A^g$ is defined via zeta regularization.
Denote by $(\gamma_1,\gamma_2,\gamma_3)_A$ the constants in Equation~\eqref{eqn:defn-Fa} determined by $A$.
It is known~\cites{BransonOrsted1991b,Branson1996} that
\begin{align*}
 (\gamma_1,\gamma_2,\gamma_3)_{L_2} & = \left( \frac{1}{8}, -\frac{1}{2}, -\frac{1}{12} \right) , \\
 (\gamma_1,\gamma_2,\gamma_3)_{\nablas^2} & = \left( \frac{7}{16}, -\frac{11}{2}, -\frac{7}{24} \right) , \\
 (\gamma_1,\gamma_2,\gamma_3)_{L_4} & = \left( -\frac{1}{4}, -14, \frac{8}{3} \right) , 
\end{align*}
where $\nablas^2$ is the square of the Dirac operator.

One motivation for this framework is the observation of Osgood, Phillips, and Sarnak~\cite{OsgoodPhillipsSarnak1988} that closed Riemannian surfaces with constant Gauss curvature extremize the functional determinant of the conformal Laplacian within their conformal class.
Branson, Chang, and Yang~\cite{BransonChangYang1992} proved this conclusion for the round four-sphere.
The techniques used to prove \cref{sobolev-dim5+,sobolev-dim3} allow us to prove this conclusion for closed locally symmetric Einstein manifolds with nonnegative scalar curvature, partially answering a question of Branson and {\O}rsted~\cite{BransonOrsted1991b}*{p.\ 676}.
More generally:

\begin{theorem}
 \label{functional-determinant-extremal}
 Let $(M^4,g)$ be a closed Riemannian four-manifold such that $\Ric = 3\lambda g \geq 0$.
 Let $\gamma_1,\gamma_2,\gamma_3 \in \bR$ be such that $\gamma_1\leq 0$ and $\gamma_2,\gamma_3 \geq 0$;
 if $\gamma_1<0$, then assume additionally that $\lv W\rv^2$ is constant.
 Then
 \begin{equation*}
  \inf \left\{ F_{\gamma_1,\gamma_2,\gamma_3}(u) \suchthat u \in C^\infty(M) \right\} = 0 .
 \end{equation*}
 Moreover, if $\max\{\gamma_2,\gamma_3\}>0$, then $u$ is extremal if and only if $e^{2u}g$ is Einstein.
\end{theorem}

\Cref{functional-determinant-extremal} applies to $S^4$, $\bCP^2$, $S^2 \times S^2$, $T^4$, and their quotients, and hence generalizes the aforementioned result of Branson, Chang, and Yang.

This article is organized as follows:
In \cref{sec:geometry} we compute the relevant divergences and set up the undetermined coefficients problem.
In \cref{sec:proof} we prove \cref{main-thm}.
In \cref{sec:operators} we discuss sufficient conditions for the nonnegativity of $I_{a,b}$ and $J$, and use these to prove \cref{main-thm-q}.
In \cref{sec:applications} we discuss the relationship between the total $I_{a,b}$-curvature and conformally invariant energy functionals, and then we prove \cref{yamabe-dim5+,yamabe-dim3,functional-determinant-extremal} and their corollaries.

\section{Some Riemannian identities}
\label{sec:geometry}

Written in terms of the trace-free part $E$ of the Schouten tensor, the $I_{a,0}$-curvature~\eqref{eqn:defn-Ia} is
\begin{equation}
 \label{eqn:Ia-via-E}
 I_{a,0} = -\Delta J - \frac{a+4}{2}\lv E\rv^2 + \frac{n^2-4+(n-1)a}{2n}J^2 .
\end{equation}

As stated in the introduction, our first task is to compute the divergences of all natural symmetric $(0,2)$-tensors which are homogeneous of degree $-4$ and which vanish when computed at Einstein metrics.
We also restrict our attention to those tensors for which the divergence is algebraically determined by the Schouten tensor $P$ and the gradient $\nabla J$ of its trace.
This leaves a three-dimensional space of tensors.

\begin{lemma}
 \label{divergences}
 Let $(M^n,g)$, $n \geq 3$, be a Riemannian manifold.
 Suppose that $a \in \bR$ is such that $I_{a,0}$ is constant.
 Then
 \begin{align*}
  \delta\left( JP - \frac{1}{n}J^2g \right) & = E(\nabla J) + \frac{n-1}{n}J\nabla J , \\
  \delta\left( (n^2-4+(n-1)a)JP - 2I_{a,0}g \right) & = (n^2-4+(n-1)a)\left( E(\nabla J) + \frac{n+1}{n}J\nabla J \right) , \\
  \delta\left( \nabla^2J - \Delta Jg \right) & = (n-2)E(\nabla J) + \frac{2(n-1)}{n}J\nabla J .
 \end{align*}
\end{lemma}

\begin{proof}
 The contracted second Bianchi identity implies that $\delta P = \nabla J$.
 The Ricci identity implies that if $u \in C^\infty(M)$, then
 \begin{equation*}
  \delta\nabla^2u = \nabla\Delta u + \Ric(\nabla u) .
 \end{equation*}
 The conclusion readily follows.
\end{proof}

Pairing these formulas against $\nabla u$ yields integral formulas involving the Hessian of $u$.
These formulas are particularly relevant when $u$ is an Einstein scale~\eqref{eqn:conformally-einstein}:

\begin{lemma}
 \label{ibp}
 Let $(M^n,g)$, $n \geq 3$, be a closed Riemannian manifold.
 Suppose that $a \in \bR$ is such that $I_{a,0}$ is constant.
 Suppose additionally that the metric $\hg := u^{-2}g$ satisfies $P^{\hg} = \frac{\lambda}{2}\hg$.
 Then
 \begin{align}
  \label{eqn:tf-JP} 0 & = \int_M \left( -uJ\lv E\rv^2 + E(\nabla J,\nabla u) + \frac{n-1}{n}\lp J\nabla J,\nabla u\rp \right) \dvol , \\
  \label{eqn:JP-Q} 0 & = \int_M \biggl( (n^2-4+(n-1)a)\Bigl( E(\nabla J,\nabla u) + \frac{(n+1)}{n}\lp J\nabla J,\nabla u\rp\Bigr) \\
   \notag & \qquad\qquad - 2uJ\Delta J - n(n+a) uJ\lv E\rv^2 + nu^{-1}\left(\lv\nabla u\rv^2 + \lambda\right)\Delta J \\
   \notag & \qquad\qquad + \frac{n(a+4)}{2}u^{-1}\left(\lv\nabla u\rv^2 + \lambda\right)\lv E\rv^2\biggr) \dvol , \\
  \label{eqn:nablaJ} 0 & = \int_M \Bigl(  (n-1)E(\nabla J,\nabla u) + \frac{n-1}{n}\lp J\nabla J,\nabla u\rp \\
   \notag & \qquad\qquad - \frac{n-1}{2}u^{-1}\left(\lv\nabla u\rv^2 + \lambda\right)\Delta J \Bigr) \dvol .
 \end{align}
\end{lemma}

\begin{proof}
 Writing the equation $P^{\hg} = \frac{\lambda}{2}\hg$ in terms of $u$ yields Equation~\eqref{eqn:conformally-einstein}.
 Since $M$ is closed, integration by parts then implies that if $T \in \Gamma(S^2T^\ast M)$, then
 \begin{equation}
  \label{eqn:pre-ibp}
  0 = \int_M \left( \lp \delta T, \nabla u\rp - u\lp T,P \rp + \frac{1}{2}u^{-1}\left(\lv\nabla u\rv^2 + \lambda\right) \tr T \right) \dvol .
 \end{equation}
 The conclusions now follow from \cref{divergences}:

 Using the formula for $\delta\bigl( JP - \frac{1}{n}J^2g\bigr)$ in Equation~\eqref{eqn:pre-ibp} yields Equation~\eqref{eqn:tf-JP}.
 
 Combining the formula for $\delta\bigl( (n^2-4+(n-1)a)JP - 2I_{a,0}g \bigr)$ in Equation~\eqref{eqn:pre-ibp} with Equation~\eqref{eqn:Ia-via-E} yields Equation~\eqref{eqn:JP-Q}.
 
 Inserting the formula for $\delta\bigl( \nabla^2J - \Delta Jg \bigr)$ into Equation~\eqref{eqn:pre-ibp} and using the identity
 \begin{equation*}
  -\int_M \left( u\lp P,\nabla^2J \rp - uJ\Delta J\right) \dvol = \int_M \left( E(\nabla J,\nabla u) - \frac{n-1}{n}\lp J\nabla J,\nabla u\rp \right) \dvol
 \end{equation*}
yields Equation~\eqref{eqn:nablaJ}.
\end{proof}

\section{Proofs of \cref{main-thm}}
\label{sec:proof}

The basic idea of the proof of \cref{main-thm} is to find a linear combination of the three identities of \cref{ibp} for which the contributions of $\Delta J$ are zero.
To that end, we first cancel the term involving $u^{-1}(\lv\nabla u\rv^2+\lambda)\Delta J$ and integrate the term $uJ\Delta J$ by parts:

\begin{lemma}
 \label{cancel-scalar-Delta}
 Let $(M^n,g)$, $n \geq 3$, be a closed Riemannian manifold.
 Suppose that $a \in \bR$ is such that $I_{a,0}$ is constant.
 Suppose additionally that the metric $\hg := u^{-2}g$ satisfies $P^{\hg} = \frac{\lambda}{2}\hg$.
 Then
 \begin{equation}
  \label{eqn:cancel-scalar-Delta}
  \begin{split}
  0 & = \int_M \biggl( 2u\lv\nabla J\rv^2 - n(n+a)uJ\lv E\rv^2 + (n^2+2n-4+(n-1)a)E(\nabla J,\nabla u) \\
   & \qquad\qquad + \frac{n^3+n^2-4+(n^2-1)a}{n}\lp J\nabla J, \nabla u\rp \\
   & \qquad\qquad + \frac{n(a+4)}{2}u^{-1}\left(\lv\nabla u\rv^2+\lambda\right)\lv E\rv^2 \biggr) \dvol .
  \end{split}
 \end{equation}
\end{lemma}

\begin{proof}
 Adding Equation~\eqref{eqn:JP-Q} to $\frac{2n}{n-1}$ times Equation~\eqref{eqn:nablaJ} yields
 \begin{align*}
  0 & = \int_M \biggl( (n^2+2n-4+(n-1)a)E(\nabla J,\nabla u) - n(n+a)uJ\lv E\rv^2 \\
   & \qquad - 2uJ\Delta J + \frac{n^3+n^2-2n-4+(n^2-1)a}{n}\lp J\nabla J,\nabla u\rp \\
   & \qquad + \frac{n(a+4)}{2}u^{-1}(\lv\nabla u\rv^2 + \lambda)\lv E\rv^2 \biggr) \dvol .
 \end{align*}
 The final conclusion follows from the identity
 \begin{equation*}
  \int_M uJ\Delta J \dvol = -\int_M \left( u\lv\nabla J\rv^2 + \lp J\nabla J,\nabla u \rp \right) \dvol . \qedhere
 \end{equation*}
\end{proof}

We now cancel the term involving $\lp J\nabla J,\nabla u\rp$.

\begin{lemma}
 \label{cancel-nabla-J2}
 Let $(M^n,g)$, $n \geq 3$, be a closed Riemannian manifold.
 Suppose that $a \in \bR$ is such that $I_{a,0}$ is constant.
 Suppose additionally that the metric $\hg := u^{-2}g$ satisfies $P^{\hg} = \frac{\lambda}{2}\hg$.
 Then
 \begin{align*}
  0 & = \int_M \biggl( 2u\lv\nabla J\rv^2 - \frac{2(3n-4+(n-1)a)}{n-1}E(\nabla J,\nabla u) \\
   & \qquad\qquad + \frac{2n^2-4+(n-1)a}{n-1}uJ\lv E\rv^2 + \frac{n(a+4)}{2}u^{-1}(\lv\nabla u\rv^2+\lambda)\lv E\rv^2 \biggr) \dvol .
 \end{align*}
\end{lemma}

\begin{proof}
 Subtract $\frac{n^3+n^2-4+(n^2-1)a}{n-1}$ times Equation~\eqref{eqn:tf-JP} from Equation~\eqref{eqn:cancel-scalar-Delta}.
\end{proof}

We are now ready to prove our first analogue of the Obata--V\'etois Theorem:

\begin{proof}[Proof of \cref{main-thm}]
 Denote
 \begin{equation*}
  C(n,a) := 4n^3 - 17n^2 + 28n - 16 + (n-1)(n^2-7n+8)a - (n-1)^2a^2 .
 \end{equation*}
 Observe that Condition~\eqref{eqn:a-range} is equivalent to the assumption $C(n,a) \geq 0$.
 Direct computation gives $C(n,-4) < 0$ and $C(n,0)>0$.
 Since $a$ satisfies Condition~\eqref{eqn:a-range}, we deduce that $a>-4$, and hence $2n^2-4+(n-1)a > 0$.

 Let $\hg = u^{-2}g$ be such that $P^{\hg} = \frac{\lambda}{2}\hg$.
 Since $n \geq 3$, the contracted second Bianchi identity implies that $\lambda$ is constant.
 Since $g$ has nonnegative scalar curvature, the Yamabe constant of $(M^n,g)$ is nonnegative.
 Thus $\lambda\geq0$, and $\lambda>0$ if $J \not= 0$.
 
 Now, the Cauchy--Schwarz Inequality implies that
 \begin{multline*}
  -\frac{2(3n-4+(n-1)a)}{n-1}E(\nabla J,\nabla u) \geq -2u\lv\nabla J\rv^2 \\
   - \frac{(3n-4+(n-1)a)^2}{2(n-1)^2}u^{-1}\lv E\rv^2\lv\nabla u\rv^2 .
 \end{multline*}
 Combining this with \cref{cancel-nabla-J2} yields
 \begin{multline*}
  0 \geq \int_M \biggl( \frac{2n^2-4+(n-1)a}{n-1}uJ\lv E\rv^2 + \frac{n(a+4)}{2}\lambda u^{-1}\lv E\rv^2 \\
   + \frac{C(n,a)}{2(n-1)^2}u^{-1}\lv\nabla u\rv^2\lv E\rv^2 \biggr) \dvol .
 \end{multline*}
 If $\lambda>0$, then $E=0$.
 If instead $\lambda=0$, then $\lv E\rv^2\lv\nabla u\rv^2=0$.
 Suppose that $p \in M$ is such that $E_p \not= 0$.
 Then $\nabla u=0$ in a neighborhood of $p$.
 Equation~\eqref{eqn:conformally-einstein} implies that $E_p=0$, a contraction.
 Thus again $E=0$.
 Therefore $g$ is Einstein.
\end{proof}

\section{Proof of \cref{main-thm-q}}
\label{sec:operators}

The first step in proving \cref{main-thm-q} is to give sufficient conditions for the $I_{a,0}$-curvature to be nonnegative under the hypotheses of \cref{main-thm}.
When $a\not=0$, we do this with the help of the $I_{-4,0}$-Yamabe constant (cf.\ \cite{BransonChangYang1992}*{Lemma~5.4}).

\begin{lemma}
 \label{DJ-yamabe}
 Let $(M^n,g)$, $n \geq 3$, be a closed Riemannian manifold such that $\Ric = (n-1)\lambda g \geq 0$.
 Then
 \begin{equation*}
  \inf_{\hg \in [g]} \left\{ \int_M (J^{\hg})^2 \dvol_{\hg} \suchthat \Vol_{\hg}(M) = 1 \right\} = J^2\Vol(M)^{\frac{4}{n}} .
 \end{equation*}
 Moreover, $\hg$ is extremal if and only if it is Einstein.
 In particular,
 \begin{align*}
  \inf_{\hg \in [g]} \left\{ \int_M I_{-4,0}^{\hg} \dvol_{\hg} \suchthat \Vol_{\hg}(M) = 1 \right\} & = \frac{n-4}{2}J^2\Vol(M)^{\frac{4}{n}} , && \text{if $n \geq 5$} , \\
  \sup_{\hg \in [g]} \left\{ \int_M I_{-4,0}^{\hg} \dvol_{\hg} \suchthat \Vol_{\hg}(M) = 1 \right\} & = -\frac{1}{2}J^2\Vol(M)^{\frac{4}{3}} , && \text{if $n = 3$} ,
 \end{align*}
 and $\hg$ is extremal if and only if it is Einstein.
\end{lemma}

\begin{proof}
 Let $\hg \in [g]$ be such that $\Vol_{\hg}(M) = 1$.
 The resolution of the Yamabe Problem~\cites{Aubin1976,Schoen1984,LeeParker1987,Trudinger1968} and Obata's Theorem~\cite{Obata1971}*{Proposition~6.2} imply that
 \begin{equation*}
  J\Vol(M)^{\frac{2}{n}} \leq \int_M J^{\hg} \dvol_{\hg}
 \end{equation*}
 with equality if and only if $\hg$ is Einstein.
 Since $J \geq 0$, squaring both sides and applying H\"older's inequality yields
 \begin{equation*}
  J^2\Vol(M)^{\frac{4}{n}} \leq \left( \int_M J^{\hg} \dvol_{\hg} \right)^2 \leq \int_M (J^{\hg})^2 \dvol_{\hg}
 \end{equation*}
 with equality if and only if $\hg$ is Einstein.
 The final conclusion uses the identity
 \begin{equation*}
  \int_M I_{-4,0}^{\hg} \dvol_{\hg} = \frac{n-4}{2}\int_M (J^{\hg})^2 \dvol_{\hg} . \qedhere
 \end{equation*}
\end{proof}

\Cref{DJ-yamabe} leads to sufficient conditions for $I_{a,0}$ to be nonnegative.

\begin{proposition}
 \label{sign}
 Let $(M^n,g)$, $n \geq 4$, be a conformally Einstein manifold with nonnegative Yamabe constant.
 Suppose additionally that $a \geq -4$ is such that $I_{a,0}$ is constant;
 if $n \geq 5$, then assume also that $a \leq 0$.
 Then $I_{a,0} \geq 0$.
\end{proposition}

\begin{proof}
 Let $g_0 \in [g]$ be such that $\Ric_{g_0} = (n-1)\lambda g_0$.
 Then $\lambda$ is a nonnegative constant.
 Moreover (cf.\ \cite{Gover2006q}*{Theorem~1.2}),
 \begin{equation}
  \label{eqn:factorization}
  \begin{split}
  Q^{g_0} & = \frac{n(n^2-4)}{8}\lambda^2 , \\
  L_4^{g_0} & = \left( -\Delta_{g_0} + \frac{n(n-2)}{4}\lambda \right)\left( -\Delta_{g_0} + \frac{(n+2)(n-4)}{4}\lambda \right) .
  \end{split}
 \end{equation}
 In particular, if $n \geq 5$, then writing $g = u^{\frac{4}{n-4}}g_0$ and using the conformal transformation law~\cite{Paneitz1983}*{Theorem~1} for the Paneitz operator yields
 \begin{equation}
  \label{eqn:paneitz-estimate}
  \int_M Q \dvol = \frac{2}{n-4}\int_M uL_4^{g_0}u \dvol_{g_0} \geq \frac{n(n^2-4)}{8}\lambda^2\int_M u^2 \dvol_{g_0} \geq 0 .
 \end{equation}
 
 Now write
 \begin{equation*}
  I_{a,0} = \frac{a+4}{4}Q - \frac{a}{4}I_{-4,0} .
 \end{equation*}
 If $n=4$, then $I_{-4,0} = -\Delta J$, and hence the conformal invariance of the total $Q$-curvature~\cite{ChangYang1995} yields
 \begin{equation*}
  I_{a,0} \Vol(M) = \int_M I_{a,0}\dvol = \frac{a+4}{4}\int_M Q^{g_0}\dvol_{g_0} \geq 0 .
 \end{equation*}
 If instead $n \geq 5$, then combining Inequality~\eqref{eqn:paneitz-estimate} with \cref{DJ-yamabe} yields
 \begin{equation*}
  I_{a,0}\Vol(M) = \int_M I_{a,0} \dvol \geq -\frac{a}{4}\int_M I_{-4,0} \dvol \geq 0 .
 \end{equation*}
 In either case we conclude that $I_{a,0} \geq 0$.
\end{proof}

Removing the assumption on the sign of the scalar curvature from \cref{main-thm} requires a generalization of an observation of Gursky~\cite{Gursky1998}*{Lemma~1.2}.

\begin{lemma}
 \label{gursky}
 Let $(M^n,g)$, $n \geq 3$, be a closed Riemannian manifold with $L_2J \geq 0$.
 \begin{enumerate}
  \item If $L_2 > 0$, then $J > 0$.
  \item If $L_2 \geq 0$ and $\ker L_2$ is nonempty, then $J=0$.
 \end{enumerate}
\end{lemma}

\begin{proof}
 Let $\lambda \geq 0$ be the first eigenvalue of $L_2$.
 A standard variational argument implies that there is a positive $u \in C^\infty(M)$ such that $L_2u = \lambda u$.
 Set $F := \frac{J}{u}$.
 Direct computation gives
 \begin{equation*}
  \Delta F + 2\lp\nabla F,\nabla\ln u\rp = -u^{-1}L_2J + \lambda F \leq \lambda F .
 \end{equation*}
 The conclusion now follows from the Strong Maximum Principle.
\end{proof}

We now prove our second version of the Obata--V\'etois Theorem:

\begin{proof}[Proof of \cref{main-thm-q}]
 Let $g_0 \in [g]$ be such that $\Ric_{g_0} = (n-1)\lambda g_0$.
 
 \underline{Case 1}: $a=0$.
 We first show that $Q \geq 0$ with equality if and only if $\lambda=0$.
 Write $g = u^{\frac{4}{n-4}}g_0$ if $n\not=4$ and $g = e^{2u}g_0$ if $n=4$.
 The conformal transformation law for the $Q$-curvature~\cite{Branson1995}*{p.\ 3679} yields
 \begin{equation}
  \label{eqn:transformation}
  \begin{aligned}
   u^{\frac{n+4}{n-4}}Q & = \frac{2}{n-4}L_4^{g_0}(u) , && \text{if $n\not=4$} , \\
   e^{4u}Q & = Q^{g_0} + L_4^{g_0}(u) , && \text{if $n=4$} .
  \end{aligned}
 \end{equation}
 Since $Q$ is constant, integrating Equation~\eqref{eqn:transformation} with respect to $\dvol_{g_0}$ and applying Equation~\eqref{eqn:factorization} yields $Q \geq 0$ with equality if and only if $\lambda=0$.
 
 We now show that $g$ is Einstein.
 If $\lambda=0$, then Equations~\eqref{eqn:factorization} and~\eqref{eqn:transformation} yield $u \in \ker L_4^{g_0} = \ker \Delta_{g_0}^2$.
 Thus $u$ is constant, and hence $g$ is Einstein.
 If instead $\lambda>0$, then a maximum principle argument~\cite{Vetois2022}*{Theorem~2.3} implies that $J > 0$.
 We then conclude from \cref{main-thm} that $g$ is Einstein.
 
 \underline{Case 2}: $a \not= 0$.
 We first observe that $I_{a,0} \geq 0$.
 If $n=3$, then this is true by assumption.
 Otherwise it follows from \cref{sign}.
 
 We now show that $g$ is Einstein.
 We deduce from Equation~\eqref{eqn:Ia-via-E} that
 \begin{equation*}
  0 \leq I_{a,0} = L_2J - \frac{a+4}{2}\lv E\rv^2 + \frac{2(n-2)+(n-1)a}{2n}J^2 \leq L_2J .
 \end{equation*}
 \Cref{gursky} now yields $J^g \geq 0$.
 \Cref{main-thm} implies that $g$ is Einstein.
\end{proof}

\section{Applications}
\label{sec:applications}

Case, Lin, and Yuan~\cite{CaseLinYuan2018b} showed that to each conformally variational scalar Riemannian invariant one can associate a formally self-adjoint, conformally covariant polydifferential operator in a way which generalizes the relationship between the $Q$-curvatures and the GJMS operators~\cite{Branson1995}.
The operators associated to the $I_{a,b}$-curvatures yield the relationship between the $I_{a,b}$-Yamabe constants and sharp Sobolev constants needed to derive \cref{sobolev-dim3,sobolev-dim5+}.

\begin{proposition}
 \label{energy}
 Let $(M^n,g)$, $n \not= 4$, be a closed Riemannian manifold such that $\Ric = (n-1)\lambda g$.
 Fix $a,b \in \bR$.
 Given $u \in C^\infty(M)$, set $g_u := u^{\frac{8}{n-4}}g$.
 Then
 \begin{align*}
  \MoveEqLeft \frac{n-4}{2}\int_M I_{a,b}^{g_u} \dvol_{g_u} = \int_M \Biggl( (\Delta u^2)^2 - \frac{16}{(n-4)^2}a\lv\nabla u\rv^4 - \frac{4}{n-4}a\lv\nabla u\rv^2\Delta u^2 \\
   & \quad + \left( \frac{n^2-2n-4}{2} + \frac{n-1}{2}a\right)\lambda\lv\nabla u^2\rv^2 \\
   & \quad + \frac{n-4}{2}\biggl( \frac{n(n^2-4)}{8}\lambda^2 + \frac{n(n-1)}{8}a\lambda^2 + b\lv W\rv^2\biggr) u^4 \Biggr) \dvol .
 \end{align*}
\end{proposition}

\begin{proof}
 First, the conformal covariance of the Weyl tensor immediately gives
 \begin{equation*}
  \int_M \lv W^{g_u} \rv_{g_u}^2 \dvol_{g_u} = \int_M \lv W \rv^2u^4 \dvol .
 \end{equation*}
 Second, the conformal transformation law~\cite{Paneitz1983}*{Theorem~1} for the Paneitz operator implies that
 \begin{equation*}
  \frac{n-4}{2}\int_M Q^{g_u} \dvol_{g_u} = \int_M u^2L_4(u^2) \dvol .
 \end{equation*}
 Third, the conformal covariance~\cite{Case2019fl}*{Theorem~2.1 and Remark~2.2} of the operator
 \begin{multline*}
  L_{\sigma_2}(u) := \frac{1}{2}\delta\left( \lv \nabla u\rv^2 \, du \right) - \frac{n-4}{16}\left( u\Delta\lv\nabla u\rv^2 - \delta\bigl( (\Delta u^2) \, du \bigr) \right) \\
   - \frac{1}{2}\left( \frac{n-4}{4} \right)^2u\delta\left( (Jg - P)(\nabla u^2) \right) + \left( \frac{n-4}{4} \right)^3\sigma_2u^3
 \end{multline*}
 implies that
 \begin{equation*}
  \left( \frac{n-4}{4} \right)^3 \int_M \sigma_2^{g_u} \dvol_{g_u} = \int_M u L_{\sigma_2}(u) \dvol .
 \end{equation*}
 Combining these formulas with the fact $P = \frac{\lambda}{2}g$ yields the desired conclusion.
\end{proof}

Computing the $Q$-Yamabe constant, and hence the $I_{a,b}$-Yamabe constants, requires separately considering the cases $n\geq 5$, $n=3$, and $n=4$.

\subsection{The case of dimension at least five}
\label{subsec:dim5}

An existence result of Gursky and Malchiodi~\cite{GurskyMalchiodi2014} allows us to compute the $Q$-Yamabe constant of a closed Einstein manifold in dimension at least five.

\begin{proposition}
 \label{q-constant-5}
 Let $(M^n,g)$, $n \geq 5$, be a closed Riemannian manifold with $\Ric = (n-1)\lambda g \geq 0$.
 Then
 \begin{equation}
  \label{eqn:q-constant-5}
  \inf_{\hg \in [g]} \left\{ \int_M Q^{\hg} \dvol_{\hg} \suchthat \Vol_{\hg}(M) = 1 \right\} = \frac{\Gamma\bigl(\frac{n+4}{2}\bigr)}{\Gamma\bigl(\frac{n-4}{2}\bigr)}\lambda^2\Vol(M)^{\frac{4}{n}} .
 \end{equation}
 Moreover, $\hg$ is extremal if and only if it is Einstein.
\end{proposition}

\begin{proof}
 If $\lambda=0$, then $L_4^g=\Delta^2$.
 Therefore
 \begin{multline*}
  \inf_{\hg \in [g]} \left\{ \int_M Q^{\hg} \dvol_{\hg} \suchthat \Vol_{\hg}(M) = 1 \right\} \\
   = \inf_{0<u\in C^\infty(M)} \left\{ \frac{2}{n-4}\int_M uL_4u \, \dvol \suchthat \int_M u^{\frac{2n}{n-4}}\dvol = 1 \right\} = 0 .
 \end{multline*}
 Moreover, $u$ extremizes $\int uL_4u \dvol$ if and only if $u$ is constant.
 
 If instead $\lambda>0$, then $g$ has positive scalar curvature and positive $Q$-curvature.
 Hence~\cite{GurskyMalchiodi2014}*{p.\ 2140} there is metric $\cg = e^{2u}g$ such that $\Vol_{\cg}(M)=1$ and
 \begin{equation*}
  Q^{\cg} = \inf_{\hg \in [g]} \left\{ \int_M Q^{\hg} \dvol_{\hg} \suchthat \Vol_{\hg}(M) = 1 \right\} .
 \end{equation*}
 \Cref{main-thm-q} implies that $\cg$ is Einstein.
 Therefore~\cite{Obata1962}*{Theorem~A} there is a diffeomorphism $\Phi \colon M \to M$ such that $\Phi^\ast\cg = cg$ for some constant $c>0$.
 Equation~\eqref{eqn:q-constant-5} readily follows.
\end{proof}

We also need to compute the $\lv W\rv^2$-Yamabe constant for certain manifolds.

\begin{lemma}
 \label{W-yamabe}
 Let $(M^n,g)$, $n \geq 3$, be a closed Riemannian manifold such that $\lv W\rv^2$ is constant.
 Then
 \begin{equation*}
  \sup_{\hg \in [g]} \left\{ \int_M \lv W^{\hg} \rv^2_{\hg} \dvol_{\hg} \suchthat \Vol_{\hg}(M) = 1 \right\} = \lv W \rv^2 \Vol(M)^{\frac{4}{n}} .
 \end{equation*}
 Moreover, $\hg$ is extremal if and only if $W=0$ or $\hg$ is homothetic to $g$.
\end{lemma}

\begin{proof}
 Let $\hg \in [g]$ be such that $\Vol_{\hg}(M)=1$.
 Define $u \in C^\infty(M)$ by $\hg = e^{2u}g$.
 Then $\int e^{nu}\dvol = 1$.
 Combining these observations with H\"older's inequality yields
 \begin{equation*}
  \int_M \lv W^{\hg} \rv_{\hg}^2 \dvol_{\hg} = \lv W\rv^2\int_M e^{(n-4)u}\dvol \leq \lv W\rv^2 \Vol(M)^{\frac{4}{n}}
 \end{equation*}
 with equality if and only if $W=0$ or $u$ is constant.
\end{proof}

Since the minimizers of the $Q$-, $I_{-4,0}$-, and $(-\lv W\rv^2)$-Yamabe constants are the same on Einstein manifolds with $\lv W\rv^2$ constant, we can compute the Yamabe-type constants of their convex combinations.

\begin{proof}[Proof of \cref{yamabe-dim5+}]
 Note that
 \begin{equation*}
  I_{a,b} = \frac{a+4}{4}Q - \frac{a}{4}I_{-4,0} + b\lv W\rv^2 .
 \end{equation*}
 Since $-4 \leq a \leq 0$ and $b\leq 0$, we immediately deduce from \cref{W-yamabe,DJ-yamabe,q-constant-5} that
 \begin{equation*}
  \inf_{\hg \in [g]} \left\{ \int_M I_{a,b}^{\hg} \dvol_{\hg} \suchthat \Vol_{\hg}(M) = 1 \right\} = I_{a,b} \Vol(M)^{\frac{4}{n}} ,
 \end{equation*}
 and moreover, that $\hg$ is extremal if and only if it is Einstein.
 The explicit value~\eqref{eqn:yamabe-dim5+-constant} follows by direct computation.
\end{proof}

Writing the total $I_{a,b}$-curvature with respect to a background Einstein metric yields our functional inequality.

\begin{proof}[Proof of \cref{sobolev-dim5+}]
 Combine \cref{yamabe-dim5+} with \cref{energy}.
\end{proof}

\subsection{The case of dimension three}
\label{subsec:dim3}

An existence result of Hang and Yang~\cite{HangYang2016t} allows us to compute the $Q$-Yamabe constant of a closed Einstein three-manifold (cf.\ \cites{HangYang2004,YangZhu2004,ChoiXu2009}).
To that end, recall that a closed three-manifold $(M^3,g)$ satisfies \defn{Condition~(NN)} if $\int uL_4u \dvol \geq 0$ for every $u \in C^\infty(M)$ such that $u^{-1}(\{0\})\not=\emptyset$.

\begin{proposition}
 \label{q-constant-3}
 Let $(M^3,g)$ be a closed Riemannian three-manifold with $\Ric = 2\lambda g \geq 0$.
 Then
 \begin{equation}
  \label{eqn:q-constant-3}
  \sup_{\hg \in [g]} \left\{ \int_M Q^{\hg} \dvol_{\hg} \suchthat \Vol_{\hg}(M) = 1 \right\} = \frac{15}{8}\lambda^2\Vol(M)^{\frac{4}{3}} .
 \end{equation}
 Moreover, $\hg$ is extremal if and only if it is Einstein.
\end{proposition}

\begin{proof}
 If $\lambda=0$, then the Paneitz operator is $L_4=\Delta^2$.
 Therefore
 \begin{multline*}
  \sup_{\hg \in [g]} \left\{ \int_M Q^{\hg} \dvol_{\hg} \suchthat \Vol_{\hg}(M) = 1 \right\} \\
   = \inf_{0<u\in C^\infty(M)} \left\{ 2\int_M (\Delta u)^2 \, \dvol \suchthat \int_M u^{-6}\dvol = 1 \right\} = 0 .
 \end{multline*}
 Moreover, $u$ extremizes $\int uL_4u \dvol$ if and only if it is constant.

 If instead $\lambda>0$, then $(M^3,g)$ is a finite quotient of the round three-sphere.
 Since the round three-sphere satisfies Condition~(NN)~\cite{HangYang2004}*{Corollary~7.1}, we conclude by a straightforward covering argument that $(M^3,g)$ also satisfies Condition~(NN).
 Therefore there is~\cite{HangYang2016t}*{Theorem~1.2} a metric $\cg = e^{2u}g$ such that
 \begin{equation*}
  Q^{\cg} = \sup_{\hg \in [g]} \left\{ \int_M Q^{\hg} \dvol_{\hg} \suchthat \Vol_{\hg}(M) = 1 \right\} \geq \frac{15}{8}\lambda^2\Vol(M)^{\frac{4}{3}} .
 \end{equation*}
 \Cref{main-thm-q} implies that $\cg$ is Einstein.
 Equation~\eqref{eqn:q-constant-3} now follows as in the proof of \cref{q-constant-5}.
\end{proof}

Since the maximizers of the $Q$- and $I_{-4,0}$--Yamabe constants are the same, the same is true of their convex combinations.

\begin{proof}[Proof of \cref{yamabe-dim3}]
 Note that
 \begin{equation*}
  I_{a,0} = \frac{a+4}{4}Q - \frac{a}{4}I_{-4,0} .
 \end{equation*}
 Since $a \in [-4,0]$, we deduce from \cref{DJ-yamabe,q-constant-3} that
 \begin{equation*}
  \sup_{\hg \in [g]} \left\{ \int_M I_{a,b}^{\hg} \dvol_{\hg} \suchthat \Vol_{\hg}(M) = 1 \right\} = I_{a,b} \Vol(M)^{\frac{4}{3}} , \\
 \end{equation*}
 and moreover, that $\hg$ is extremal if and only if it is Einstein.
 The explicit value~\eqref{eqn:yamabe-dim3-constant} follows by direct computation.
\end{proof}

Writing the total $I_{a,b}$-curvature with respect to a background Einstein metric yields our functional inequality.

\begin{proof}[Proof of \cref{sobolev-dim3}]
 Combine \cref{yamabe-dim3} with \cref{energy}.
\end{proof}

\subsection{The case of dimension four}
\label{subsec:dim4}

In dimension four we directly compute the infimum of the $II$-functional.

\begin{proposition}
 \label{q-constant-4}
 Let $(M^4,g)$ be a closed Riemannian four-manifold with $\Ric = 2\lambda g \geq 0$.
 Then
 \begin{equation*}
  \inf \left\{ II(u) \suchthat u \in C^\infty(M) \right\} = 0 .
 \end{equation*}
 Moreover, $u$ is extremal if and only if $e^{2u}g$ is Einstein.
\end{proposition}

\begin{proof}
 We directly compute that $Q = 6\lambda^2$ and $L_4=-\Delta(-\Delta + 2)$.
 
 Suppose first that $\lambda = 0$.
 Then
 \begin{equation*}
  II(u) = \int_M uL_4u \dvol \geq 0
 \end{equation*}
 with equality if and only if $u$ is constant.
 
 Suppose now that $\lambda > 0$.
 Clearly $L_4 \geq 0$ with $\ker L_4$ equal to the constant functions.
 A result of Gursky~\cite{Gursky1999}*{Theorem~B} implies that $\int Q \dvol \leq 16\pi^2$ with equality if and only if $(M^4,g)$ is conformal to the round four-sphere.
 Existence results of Beckner~\cite{Beckner1993}*{Theorem~1} and of Chang and Yang~\cite{ChangYang1995}*{Theorem~1.2} yield a minimizer $v \in C^\infty(M)$ for $II$, and necessarily $e^{2v}g$ has constant $Q$-curvature.
 We deduce from \cref{main-thm-q} that $e^{2v}g$ is constant and hence---applying conformal invariance~\cite{Beckner1993}*{Theorem~1} if $(M^4,g)$ is conformally equivalent to $(S^4,g)$---that~$II(v)=0$.
\end{proof}

The classification of extremals of functional determinants again follows from a convexity argument.

\begin{proof}[Proof of \cref{functional-determinant-extremal}]
 Recall~\cite{ChangYang1995}*{p.\ 173} that
 \begin{equation*}
  III(u) = 12\left[ \int_M (J^{g_u})^2 \dvol_{g_u} - \int_M (J^g)^2 \dvol_{g} \right] ,
 \end{equation*}
 where $g_u := e^{2u}g$.
 \Cref{DJ-yamabe} thus gives $III(u) \geq 0$ with equality if and only if $g_u$ is Einstein.
 Since $\lv W\rv^2$ is constant, we deduce from Jensen's inequality that $I(u) \leq 0$ with equality if and only if $u$ is constant.
 Combining these observations with \cref{q-constant-4} yields the final conclusion.
\end{proof}

\section*{Acknowledgements} I would like to thank Matthew Gursky for helpful discussions about V\'etois' argument and critical points of the functional determinant.
I would also like to thank Yueh-Ju Lin for helpful discussions about higher-dimensional applications.
This work was partially supported by the Simons Foundation (Grant \#524601), and by the Simons Foundation and the Mathematisches Forschungsinstitut Oberwolfach via a Simons Visiting Professorship.

\bibliography{bib}

\newcommand{\noopsort}[1]{}
% \bib, bibdiv, biblist are defined by the amsrefs package.
\begin{bibdiv}
\begin{biblist}

\bib{Aubin1976}{article}{
      author={Aubin, Thierry},
       title={\'{E}quations diff\'{e}rentielles non lin\'{e}aires et probl\`eme de {Y}amabe concernant la courbure scalaire},
        date={1976},
        ISSN={0021-7824},
     journal={J. Math. Pures Appl. (9)},
      volume={55},
      number={3},
       pages={269\ndash 296},
      review={\MR{431287}},
}

\bib{Beckner1993}{article}{
      author={Beckner, William},
       title={Sharp {S}obolev inequalities on the sphere and the {M}oser-{T}rudinger inequality},
        date={1993},
        ISSN={0003-486X},
     journal={Ann. of Math. (2)},
      volume={138},
      number={1},
       pages={213\ndash 242},
         url={http://dx.doi.org/10.2307/2946638},
      review={\MR{1230930}},
}

\bib{Bidaut-VeronVeron1991}{article}{
      author={Bidaut-V\'{e}ron, Marie-Fran\c{c}oise},
      author={V\'{e}ron, Laurent},
       title={Nonlinear elliptic equations on compact {R}iemannian manifolds and asymptotics of {E}mden equations},
        date={1991},
        ISSN={0020-9910},
     journal={Invent. Math.},
      volume={106},
      number={3},
       pages={489\ndash 539},
         url={https://doi.org/10.1007/BF01243922},
      review={\MR{1134481}},
}

\bib{Branson1995}{article}{
      author={Branson, Thomas~P.},
       title={Sharp inequalities, the functional determinant, and the complementary series},
        date={1995},
        ISSN={0002-9947},
     journal={Trans. Amer. Math. Soc.},
      volume={347},
      number={10},
       pages={3671\ndash 3742},
         url={https://doi.org/10.2307/2155203},
      review={\MR{1316845}},
}

\bib{Branson1996}{article}{
      author={Branson, Thomas~P.},
       title={An anomaly associated with {$4$}-dimensional quantum gravity},
        date={1996},
        ISSN={0010-3616},
     journal={Comm. Math. Phys.},
      volume={178},
      number={2},
       pages={301\ndash 309},
         url={http://projecteuclid.org/euclid.cmp/1104286653},
      review={\MR{1389906}},
}

\bib{BransonChangYang1992}{article}{
      author={Branson, Thomas~P.},
      author={Chang, Sun-Yung~A.},
      author={Yang, Paul~C.},
       title={Estimates and extremals for zeta function determinants on four-manifolds},
        date={1992},
        ISSN={0010-3616},
     journal={Comm. Math. Phys.},
      volume={149},
      number={2},
       pages={241\ndash 262},
         url={http://projecteuclid.org/getRecord?id=euclid.cmp/1104251221},
      review={\MR{1186028}},
}

\bib{BransonOrsted1988}{article}{
      author={Branson, Thomas~P.},
      author={{\O}rsted, Bent},
       title={Conformal deformation and the heat operator},
        date={1988},
        ISSN={0022-2518},
     journal={Indiana Univ. Math. J.},
      volume={37},
      number={1},
       pages={83\ndash 110},
         url={http://dx.doi.org/10.1512/iumj.1988.37.37004},
      review={\MR{942096}},
}

\bib{BransonOrsted1991b}{article}{
      author={Branson, Thomas~P.},
      author={{\O}rsted, Bent},
       title={Explicit functional determinants in four dimensions},
        date={1991},
        ISSN={0002-9939},
     journal={Proc. Amer. Math. Soc.},
      volume={113},
      number={3},
       pages={669\ndash 682},
         url={http://dx.doi.org/10.2307/2048601},
      review={\MR{1050018}},
}

\bib{Case2019fl}{article}{
      author={Case, Jeffrey~S.},
       title={The {F}rank-{L}ieb approach to sharp {S}obolev inequalities},
        date={2021},
        ISSN={0219-1997},
     journal={Commun. Contemp. Math.},
      volume={23},
      number={3},
       pages={Paper No. 2050015, 16},
         url={https://doi.org/10.1142/S0219199720500157},
      review={\MR{4216420}},
}

\bib{CaseLinYuan2016}{article}{
      author={Case, Jeffrey~S.},
      author={Lin, Yueh-Ju},
      author={Yuan, Wei},
       title={Conformally variational {R}iemannian invariants},
        date={2019},
        ISSN={0002-9947},
     journal={Trans. Amer. Math. Soc.},
      volume={371},
      number={11},
       pages={8217\ndash 8254},
         url={https://doi.org/10.1090/tran/7761},
      review={\MR{3955546}},
}

\bib{CaseLinYuan2018b}{article}{
      author={Case, Jeffrey~S.},
      author={Lin, Yueh-Ju},
      author={Yuan, Wei},
       title={Some constructions of formally self-adjoint conformally covariant polydifferential operators},
        date={2022},
        ISSN={0001-8708},
     journal={Adv. Math.},
      volume={401},
       pages={Paper No. 108312},
         url={https://doi.org/10.1016/j.aim.2022.108312},
      review={\MR{4392224}},
}

\bib{ChangGurskyYang2003b}{incollection}{
      author={Chang, Sun-Yung~A.},
      author={Gursky, Matthew~J.},
      author={Yang, Paul~C.},
       title={Entire solutions of a fully nonlinear equation},
        date={2003},
   booktitle={Lectures on partial differential equations},
      series={New Stud. Adv. Math.},
      volume={2},
   publisher={Int. Press, Somerville, MA},
       pages={43\ndash 60},
      review={\MR{2055838}},
}

\bib{ChangYang1995}{article}{
      author={Chang, Sun-Yung~A.},
      author={Yang, Paul~C.},
       title={Extremal metrics of zeta function determinants on {$4$}-manifolds},
        date={1995},
        ISSN={0003-486X},
     journal={Ann. of Math. (2)},
      volume={142},
      number={1},
       pages={171\ndash 212},
         url={http://dx.doi.org/10.2307/2118613},
      review={\MR{1338677}},
}

\bib{ChoiXu2009}{article}{
      author={Choi, Y.~S.},
      author={Xu, Xingwang},
       title={Nonlinear biharmonic equations with negative exponents},
        date={2009},
        ISSN={0022-0396,1090-2732},
     journal={J. Differential Equations},
      volume={246},
      number={1},
       pages={216\ndash 234},
         url={https://doi.org/10.1016/j.jde.2008.06.027},
      review={\MR{2467021}},
}

\bib{Escobar1988}{article}{
      author={Escobar, Jos{\'e}~F.},
       title={Sharp constant in a {S}obolev trace inequality},
        date={1988},
        ISSN={0022-2518},
     journal={Indiana Univ. Math. J.},
      volume={37},
      number={3},
       pages={687\ndash 698},
         url={http://dx.doi.org/10.1512/iumj.1988.37.37033},
      review={\MR{962929}},
}

\bib{GidasSpruck1981}{article}{
      author={Gidas, B.},
      author={Spruck, J.},
       title={Global and local behavior of positive solutions of nonlinear elliptic equations},
        date={1981},
        ISSN={0010-3640,1097-0312},
     journal={Comm. Pure Appl. Math.},
      volume={34},
      number={4},
       pages={525\ndash 598},
         url={https://doi.org/10.1002/cpa.3160340406},
      review={\MR{615628}},
}

\bib{Gover2006q}{article}{
      author={Gover, A.~Rod},
       title={Laplacian operators and {$Q$}-curvature on conformally {E}instein manifolds},
        date={2006},
        ISSN={0025-5831},
     journal={Math. Ann.},
      volume={336},
      number={2},
       pages={311\ndash 334},
         url={http://dx.doi.org/10.1007/s00208-006-0004-z},
      review={\MR{2244375}},
}

\bib{Gursky1997}{article}{
      author={Gursky, Matthew~J.},
       title={Uniqueness of the functional determinant},
        date={1997},
        ISSN={0010-3616},
     journal={Comm. Math. Phys.},
      volume={189},
      number={3},
       pages={655\ndash 665},
         url={https://doi.org/10.1007/s002200050223},
      review={\MR{1482933}},
}

\bib{Gursky1998}{article}{
      author={Gursky, Matthew~J.},
       title={The {W}eyl functional, de {R}ham cohomology, and {K}\"ahler-{E}instein metrics},
        date={1998},
        ISSN={0003-486X},
     journal={Ann. of Math. (2)},
      volume={148},
      number={1},
       pages={315\ndash 337},
         url={http://dx.doi.org/10.2307/120996},
      review={\MR{1652920}},
}

\bib{Gursky1999}{article}{
      author={Gursky, Matthew~J.},
       title={The principal eigenvalue of a conformally invariant differential operator, with an application to semilinear elliptic {PDE}},
        date={1999},
        ISSN={0010-3616},
     journal={Comm. Math. Phys.},
      volume={207},
      number={1},
       pages={131\ndash 143},
         url={http://dx.doi.org/10.1007/s002200050721},
      review={\MR{1724863}},
}

\bib{GurskyMalchiodi2012}{article}{
      author={Gursky, Matthew~J.},
      author={Malchiodi, Andrea},
       title={Non-uniqueness results for critical metrics of regularized determinants in four dimensions},
        date={2012},
        ISSN={0010-3616},
     journal={Comm. Math. Phys.},
      volume={315},
      number={1},
       pages={1\ndash 37},
         url={http://dx.doi.org/10.1007/s00220-012-1535-7},
      review={\MR{2966939}},
}

\bib{GurskyMalchiodi2014}{article}{
      author={Gursky, Matthew~J.},
      author={Malchiodi, Andrea},
       title={A strong maximum principle for the {P}aneitz operator and a non-local flow for the {$Q$}-curvature},
        date={2015},
        ISSN={1435-9855},
     journal={J. Eur. Math. Soc. (JEMS)},
      volume={17},
      number={9},
       pages={2137\ndash 2173},
         url={http://dx.doi.org/10.4171/JEMS/553},
      review={\MR{3420504}},
}

\bib{HangYang2004}{article}{
      author={Hang, Fengbo},
      author={Yang, Paul~C.},
       title={The {S}obolev inequality for {P}aneitz operator on three manifolds},
        date={2004},
        ISSN={0944-2669},
     journal={Calc. Var. Partial Differential Equations},
      volume={21},
      number={1},
       pages={57\ndash 83},
         url={https://doi.org/10.1007/s00526-003-0247-4},
      review={\MR{2078747}},
}

\bib{HangYang2016t}{article}{
      author={Hang, Fengbo},
      author={Yang, Paul~C.},
       title={{$Q$} curvature on a class of 3-manifolds},
        date={2016},
        ISSN={0010-3640},
     journal={Comm. Pure Appl. Math.},
      volume={69},
      number={4},
       pages={734\ndash 744},
         url={https://doi.org/10.1002/cpa.21559},
      review={\MR{3465087}},
}

\bib{JerisonLee1988}{article}{
      author={Jerison, David},
      author={Lee, John~M.},
       title={Extremals for the {S}obolev inequality on the {H}eisenberg group and the {CR} {Y}amabe problem},
        date={1988},
        ISSN={0894-0347},
     journal={J. Amer. Math. Soc.},
      volume={1},
      number={1},
       pages={1\ndash 13},
         url={http://dx.doi.org/10.2307/1990964},
      review={\MR{924699}},
}

\bib{LeeParker1987}{article}{
      author={Lee, John~M.},
      author={Parker, Thomas~H.},
       title={The {Y}amabe problem},
        date={1987},
        ISSN={0273-0979},
     journal={Bull. Amer. Math. Soc. (N.S.)},
      volume={17},
      number={1},
       pages={37\ndash 91},
      review={\MR{MR888880}},
}

\bib{Obata1962}{article}{
      author={Obata, Morio},
       title={Certain conditions for a {R}iemannian manifold to be isometric with a sphere},
        date={1962},
        ISSN={0025-5645},
     journal={J. Math. Soc. Japan},
      volume={14},
       pages={333\ndash 340},
      review={\MR{0142086}},
}

\bib{Obata1971}{article}{
      author={Obata, Morio},
       title={The conjectures on conformal transformations of {R}iemannian manifolds},
        date={1971/72},
        ISSN={0022-040X},
     journal={J. Differential Geometry},
      volume={6},
       pages={247\ndash 258},
      review={\MR{MR0303464}},
}

\bib{OsgoodPhillipsSarnak1988}{article}{
      author={Osgood, B.},
      author={Phillips, R.},
      author={Sarnak, P.},
       title={Extremals of determinants of {L}aplacians},
        date={1988},
        ISSN={0022-1236},
     journal={J. Funct. Anal.},
      volume={80},
      number={1},
       pages={148\ndash 211},
         url={http://dx.doi.org/10.1016/0022-1236(88)90070-5},
      review={\MR{960228}},
}

\bib{Paneitz1983}{article}{
      author={Paneitz, Stephen~M.},
       title={A quartic conformally covariant differential operator for arbitrary pseudo-{R}iemannian manifolds (summary)},
        date={2008},
        ISSN={1815-0659},
     journal={SIGMA Symmetry Integrability Geom. Methods Appl.},
      volume={4},
       pages={Paper 036, 3},
         url={http://dx.doi.org/10.3842/SIGMA.2008.036},
      review={\MR{2393291}},
}

\bib{Schoen1984}{article}{
      author={Schoen, Richard},
       title={Conformal deformation of a {R}iemannian metric to constant scalar curvature},
        date={1984},
        ISSN={0022-040X},
     journal={J. Differential Geom.},
      volume={20},
      number={2},
       pages={479\ndash 495},
      review={\MR{MR788292}},
}

\bib{Trudinger1968}{article}{
      author={Trudinger, Neil~S.},
       title={Remarks concerning the conformal deformation of {R}iemannian structures on compact manifolds},
        date={1968},
     journal={Ann. Scuola Norm. Sup. Pisa (3)},
      volume={22},
       pages={265\ndash 274},
      review={\MR{MR0240748}},
}

\bib{Vetois2022}{unpublished}{
      author={V{\'e}tois, J{\'e}r{\^o}me},
       title={Uniqueness of conformal metrics with constant {Q}-curvature on closed {E}instein manifolds},
        date={\noopsort{2099}preprint},
        note={arXiv:2210.07444},
}

\bib{Viaclovsky2000}{article}{
      author={Viaclovsky, Jeff~A.},
       title={Conformal geometry, contact geometry, and the calculus of variations},
        date={2000},
        ISSN={0012-7094},
     journal={Duke Math. J.},
      volume={101},
      number={2},
       pages={283\ndash 316},
         url={http://dx.doi.org/10.1215/S0012-7094-00-10127-5},
      review={\MR{1738176}},
}

\bib{YangZhu2004}{article}{
      author={Yang, Paul},
      author={Zhu, Meijun},
       title={On the {P}aneitz energy on standard three sphere},
        date={2004},
        ISSN={1292-8119,1262-3377},
     journal={ESAIM Control Optim. Calc. Var.},
      volume={10},
      number={2},
       pages={211\ndash 223},
         url={https://doi.org/10.1051/cocv:2004002},
      review={\MR{2083484}},
}

\end{biblist}
\end{bibdiv}
\end{document}